\documentclass[a4paper,twoside]{article}
\usepackage{a4}
\usepackage{amssymb}
\usepackage{amsmath}
\usepackage{upref}
\usepackage[active]{srcltx}
\usepackage[pagebackref,colorlinks,citecolor=blue,linkcolor=blue]{hyperref}
\usepackage[dvipsnames]{color}
\allowdisplaybreaks[2] 
\newcount\minutes \newcount\hours
\hours=\time
\divide\hours 60
\minutes=\hours
\multiply\minutes -60
\advance\minutes \time
\newcommand{\klockan}{\the\hours:{\ifnum\minutes<10 0\fi}\the\minutes}
\newcommand{\tid}{\today\ \klockan}
\newcommand{\prtid}{\smash{\raise 10mm \hbox{\LaTeX ed \tid}}}
\renewcommand{\prtid}{}
\makeatletter
\def\sectionmark#1{} 
\def\subsectionmark#1{}
\newcommand{\sectnr}{\ifnum \c@secnumdepth >\z@
                 \thesection.\hskip 1em\relax \fi}
\def\@evenhead{\footnotesize\rm\thepage\hfil\leftmark\hfil\llap{\prtid}}
\def\@oddhead{\footnotesize\rm\rlap{\prtid}\hfil\rightmark\hfil\thepage}
\def\tableofcontents{\section*{Contents} 
 \@starttoc{toc}}
\makeatother
\makeatletter
\def\@biblabel#1{#1.}
\makeatother
\makeatletter
\let\Thebibliography=\thebibliography
\renewcommand{\thebibliography}[1]{\def\@mkboth##1##2{}\Thebibliography{#1}
\addcontentsline{toc}{section}{References}
\frenchspacing 
\setlength{\@topsep}{0pt}
\setlength{\itemsep}{0pt}%
\setlength{\parskip}{0pt plus 2pt}%
}
\makeatother
\makeatletter
\def\mdots@{\mathinner.\nonscript\!.%
 \ifx\next,.\else\ifx\next;.\else\ifx\next..\else
 \nonscript\!\mathinner.\fi\fi\fi}
\let\ldots\mdots@
\let\cdots\mdots@
\let\dotso\mdots@
\let\dotsb\mdots@
\let\dotsm\mdots@
\let\dotsc\mdots@
\def\vdots{\vbox{\baselineskip2.8\p@ \lineskiplimit\z@
    \kern6\p@\hbox{.}\hbox{.}\hbox{.}\kern3\p@}}
\def\ddots{\mathinner{\mkern1mu\raise8.6\p@\vbox{\kern7\p@\hbox{.}}%
    \raise5.8\p@\hbox{.}\raise3\p@\hbox{.}\mkern1mu}}
\makeatother
\makeatletter
\let\Enumerate=\enumerate
\renewcommand{\enumerate}{\Enumerate%
\setlength{\@topsep}{0pt}
\setlength{\itemsep}{0pt}%
\setlength{\parskip}{0pt plus 1pt}%
\renewcommand{\theenumi}{\textup{(\alph{enumi})}}%
\renewcommand{\labelenumi}{\theenumi}%
}
\let\endEnumerate=\endenumerate
\renewcommand{\endenumerate}{\endEnumerate\unskip}
\makeatother
\makeatletter
\def\@seccntformat#1{\csname the#1\endcsname.\quad}
\makeatother
\newcommand{\authortitle}[2]{\author{#1}\title{#2}\markboth{#1}{#2}}
\newcommand{\auth}[2]{{#1, #2.}}
\newcommand{\art}[6]{{\sc #1, \rm #2, \it #3\/ \bf #4 \rm (#5), \mbox{#6}.}}

\newcommand{\arttoappear}[3]{{\sc #1, \rm #2, to appear in \it #3}}
\newcommand{\book}[3]{{\sc #1, \it #2, \rm #3.}}
\newcommand{\AND}{{\rm and }}
\newcommand{\private}[2]{\sc #1, \rm \emph{Private communication}, #2.}
\RequirePackage{amsthm}
\newtheoremstyle{propositional}%
  {\topsep}   
  {\topsep}   
  {\itshape}  
  {}          
  {\bfseries} 
  {.}         
  { }         
  {}          
\theoremstyle{propositional}
\newtheorem{thm}{Theorem}[section]
\newtheorem{lem}[thm]{Lemma}
\newtheorem{cor}[thm]{Corollary}
\makeatletter
\renewenvironment{proof}[1][\proofname]{\par
  \pushQED{\qed}%
  \normalfont
  \trivlist
  \item[\hskip\labelsep
        \itshape
    #1\@addpunct{.}]\ignorespaces
}{%
  \popQED\endtrivlist\@endpefalse
}
\makeatother
\newcommand{\setm}{\setminus}
\newcommand{\Cp}{C_p}
\DeclareMathOperator{\capp}{cap}
\newcommand{\cp}{\capp_p}
\newcommand{\CpU}{{C_p^U}}
\newcommand{\loc}{_{\rm loc}}
\newcommand{\ga}{\gamma}
\newcommand{\eps}{\varepsilon}
\newcommand{\la}{\lambda}
\newcommand{\p}{{$p\mspace{1mu}$}}
\newcommand{\R}{\mathbf{R}}
\newcommand{\limplus}{{\mathchoice{\vcenter{\hbox{$\scriptstyle +$}}}
  {\vcenter{\hbox{$\scriptstyle +$}}}
  {\vcenter{\hbox{$\scriptscriptstyle +$}}}
  {\vcenter{\hbox{$\scriptscriptstyle +$}}}
}}
\newcommand{\limminus}{{\mathchoice{\vcenter{\hbox{$\scriptstyle -$}}}
  {\vcenter{\hbox{$\scriptstyle -$}}}
  {\vcenter{\hbox{$\scriptscriptstyle -$}}}
  {\vcenter{\hbox{$\scriptscriptstyle -$}}}
}}
\newcommand{\limpm}{{\mathchoice{\vcenter{\hbox{$\scriptstyle \pm$}}}
  {\vcenter{\hbox{$\scriptstyle \pm$}}}
  {\vcenter{\hbox{$\scriptscriptstyle \pm$}}}
  {\vcenter{\hbox{$\scriptscriptstyle \pm$}}}
}}
\newcommand{\Np}{N^{1,p}}
\newcommand{\Nploc}{N^{1,p}\loc}
\DeclareMathOperator{\Mod}{Mod}
\newcommand{\Modp}{{\Mod_p}}
\newcommand{\Uplus}{U_\limplus}
\newcommand{\Uminus}{U_\limminus}
\newcommand{\Upm}{U_\limpm}
\newcommand{\Wp}{W^{1,p}}
\newcommand{\Wploc}{W^{1,p}\loc}
\newcommand{\Npploc}{N^{1,p}_{\textup{fine-loc}}}
\newcommand{\Lpploc}{L^{p}_{\textup{fine-loc}}}
\newcommand{\Npqloc}{N^{1,p}_{\textup{quasi-loc}}}
\newcommand{\hNpploc}{\widehat{N}^{1,p}_{\textup{fine-loc}}}

\newcommand{\eqv}{\ensuremath{
\mathchoice{\quad \Longleftrightarrow \quad}{\Leftrightarrow}
                {\Leftrightarrow}{\Leftrightarrow}} }
\newcommand{\imp}{\mathchoice{\quad \Longrightarrow \quad}{\Rightarrow}
                {\Rightarrow}{\Rightarrow}}
\numberwithin{equation}{section}
\newenvironment{ack}{\medskip{\it Acknowledgement.}}{}

\begin{document}

\authortitle{Anders Bj\"orn and Jana Bj\"orn}
{A uniqueness result for functions with zero fine gradient 
on quasiconnected  sets}
\title{A uniqueness result for functions with zero fine gradient 
on quasiconnected and finely connected sets}
\author{
Anders Bj\"orn \\
\it\small Department of Mathematics, Link\"oping University, \\
\it\small SE-581 83 Link\"oping, Sweden\/{\rm ;}
\it \small anders.bjorn@liu.se
\\
\\
Jana Bj\"orn \\
\it\small Department of Mathematics, Link\"oping University, \\
\it\small SE-581 83 Link\"oping, Sweden\/{\rm ;}
\it \small jana.bjorn@liu.se
}

\date{} 

\maketitle

\noindent{\small {\bf Abstract}. 
We show that every Sobolev function in $\Wploc(U)$  on a 
\p-quasiopen set $U \subset \R^n$ with a.e.-vanishing \p-fine gradient  
is a.e.-constant if and only if $U$ is \p-quasiconnected.
To prove this we use the theory of Newtonian Sobolev spaces
on metric measure spaces, and obtain the corresponding
equivalence
also for complete metric spaces equipped with a doubling
measure supporting a \p-Poincar\'e inequality.
On unweighted $\R^n$, we also obtain the corresponding 
result for  \p-finely open sets in terms of \p-fine connectedness, using
a deep result by Latvala. 


}

\section{Introduction}

One of the basic properties of derivatives and gradients it that they control
the oscillation, so that every sufficiently nice function
with vanishing gradient in an open connected set must be constant therein.
This is used in many proofs and holds in rather general situations, such
as for distributions (H\"ormander~\cite[Theorem~3.1.4]{hormanderI})
and Sobolev functions, 
including those on weighted $\R^n$ with a \p-admissible measure
(Heinonen--Kilpel\"ainen--Martio~\cite[Lemma~1.16]{HeKiMa}).

In this note we address a similar question on quasiopen and finely open sets
in the context of the corresponding Sobolev spaces.
Such sets and spaces are fundamental in the fine potential theory
as well as for fine properties of 
solutions of
various
partial differential equations,
see e.g.\ Mal\'y--Ziemer~\cite{MZ} and the references therein.

We will be able to prove our results in the setting of metric 
spaces, under the usual assumptions, but this is not the primary goal 
of this note.
This general approach demonstrates the strength of the metric space 
theory, since the proofs below turn out to be very natural.
In particular, arguments using upper gradients and curve families
play a crucial role. 
We will also 
rely on several recent results
from the fine potential theory on metric spaces.

Let $1<p< \infty$.
A set $U \subset \R^n$ is \emph{\p-quasiopen} if for every $\eps>0$ there is 
an open set $G$ such that $G \cup U$ is open and 
$\Cp(G)<\eps$,
where $\Cp$ is the Sobolev \p-capacity associated with the Sobolev
space $\Wp(\R^n)$.
(This can equivalently be defined using various other related capacities.)
The study of Sobolev spaces and nonlinear partial differential equations
on \p-quasiopen sets was initiated by Kilpel\"ainen and Mal\'y in~\cite{KiMa92};
see the introductions in \cite{BBLat3} and \cite{BBLat2} for more on the 
history.

The \p-quasiopen sets are preserved  under taking 
finite intersections
and countable unions, but not under arbitrary unions.
(For example, since points in $\R^n$ have zero \p-capacity, 
$1 < p \le n$, and are 
thus \p-quasiopen,  every set is a union of \p-quasiopen sets, 
but not all sets are \p-quasiopen.)

Since the \p-quasiopen sets \emph{do not} form a topology,
it is not completely obvious how to define \p-quasiconnectedness,
and there seem to be at least two natural definitions.
Following Fuglede~\cite[p.\ 164]{Fugl71} we say
that a \p-quasiopen set $U$  is \emph{\p-quasiconnected} 
if the only subsets of $U$, that
are both \p-quasiopen and relatively
\p-quasiclosed (i.e.\ 
their complement within $U$ is also \p-quasiopen),
 are the sets with zero \p-capacity and their complements (within $U$).
This definition was also used by Adams--Lewis~\cite{AdLew} in
the nonlinear potential theory.
Equivalently, $U$ is  \p-quasiconnected
if it cannot  be written as a union of two 
disjoint \p-quasiopen sets with positive \p-capacity.

The \p-quasiopen sets are closely related to \p-finely open sets,
which are defined using the Wiener type integral \eqref{deff-thin} and
form 
the coarsest topology making all \p-superharmonic functions continuous;
called the \emph{\p-fine topology}.
More precisely, $U$ is \p-quasiopen if and only if it can 
be written as a union $U=V \cup E$, where $V$ is \p-finely open and 
$\Cp(E)=0$.
Another recent characterization of \p-quasiopen sets is
that they are precisely the \p-path open sets,
see Bj\"orn--Bj\"orn--Mal\'y~\cite[Theorem~1.1]{BBMaly} and 
Shanmugalingam~\cite[Remark~3.5]{Sh-rev}.

With this connection to \p-finely open sets in mind it seems
natural to say that a \p-quasiopen set $U$ is 
\emph{weakly \p-quasiconnected} if
it can be written as a union $U=V \cup E$, where $\Cp(E)=0$
and $V$ is a \p-finely connected \p-finely open set 
(i.e.\ connected in the \p-fine topology).
A consequence is that a \p-finely open set
is \p-finely connected if and only if it is weakly
\p-quasiconnected; 
the nontrivial ``if'' part follows from Lemma~\ref{lem-Lat-converse}.

The following is our main result.

\begin{thm} \label{thm-main}
Assume that $U \subset \R^n$ is a \p-quasiopen set in
unweighted\/ $\R^n$.
Then the following are equivalent\/\textup{:}
\begin{enumerate}
\item \label{a-gu}
If $u \in \Wploc(U)$ and $\nabla u =0$ a.e.,
then there is a constant $c$ such that $u=c$ a.e.\ in $U$.
\item \label{a-conn}
$U$ is \p-quasiconnected.
\item \label{a-weakconn}
$U$ is weakly \p-quasiconnected.
\item \label{a-finconn}
If $U=V \cup E$, where $V$ is \p-finely open and $\Cp(E)=0$,
then $V$ is \p-finely connected.
\end{enumerate}
\end{thm}

Here, $\Wploc(U)$ is the local Sobolev space
defined by Kilpel\"ainen--Mal\'y~\cite{KiMa92} 
for \p-quasiopen sets $U$, and $\nabla u$ is 
the \p-fine gradient also introduced in \cite{KiMa92}.
On open sets, these notions coincide with the usual Sobolev space 
and weak (or distributional) gradient, respectively.

If $U$ is open and $u \in \Wploc(U)$ with $\nabla u =0$ a.e.,
then it follows from the \p-Poincar\'e inequality \eqref{eq-def-PI} that
$u$ is locally a.e.-constant, and hence 
necessarily a.e.-constant if and only if $U$ is connected,
i.e.\ \ref{a-gu} is equivalent to $U$ being connected in this case.
(One direction of this is \cite[Lemma~1.16]{HeKiMa}.)
Hence an \emph{open} set is connected if and only if it is \p-quasiconnected.
We thus recover Corollary~1 in Adams--Lewis~\cite{AdLew} for  
first-order Sobolev spaces;
which is the ``only if'' part of this equivalence, 
and which they obtain also for 
higher-order Sobolev spaces
in unweighted $\R^n$.
Similar facts are true also in metric spaces.

To prove Theorem~\ref{thm-main} we will use the theory of Newtonian Sobolev spaces
on metric measure spaces,
including several recent results in the fine potential theory on metric spaces.
In fact, in general metric spaces (assuming the rather standard
assumptions of completeness, doubling and a \p-Poincar\'e inequality),
we show that \ref{a-gu} $\eqv$ \ref{a-conn} $\eqv$ \ref{a-finconn}
$\imp$ \ref{a-weakconn}. (The statement \ref{a-gu} needs to be slightly 
reformulated, see Theorem~\ref{thm-main-X}.)
On the other hand, the implication \ref{a-weakconn} $\imp$ \ref{a-conn}
is equivalent to a statement about \p-fine connectedness, whose truth 
on unweighted $\R^n$ 
follows from a deep result by Latvala~\cite{Lat2000}.
We do not know whether Latvala's result can be generalized to metric spaces,
or even to weighted $\R^n$.

For \p-finely open sets, (part of) Theorem~\ref{thm-main} takes the following
form.

\begin{thm} \label{thm-main-fine}
Assume that $V \subset \R^n$ is a \p-finely open set in
unweighted\/ $\R^n$.
Then the following are equivalent\/\textup{:}
\begin{enumerate}
\item \label{b-gu}
If $u \in \Wploc(V)$ and $\nabla u =0$ a.e.,
then there is a constant $c$ such that $u=c$ a.e.\ in $V$.
\item \label{b-conn}
$V$ is \p-quasiconnected.
\item \label{b-finconn}
$V$ is \p-finely connected.
\end{enumerate}
\end{thm}

On metric spaces we know that 
\ref{b-gu} $\eqv$ \ref{b-conn} $\imp$ \ref{b-finconn},
but whether \ref{b-finconn} $\imp$ \ref{b-conn} remains an open
question.

\begin{ack}
This note has been
triggered by a question from 
Nicola Fusco~\cite{fusco-private}
on the validity of the implication 
\ref{a-weakconn} $\imp$ \ref{a-gu} in  Theorem~\ref{thm-main}.
The authors were supported by the Swedish Research Council, 
grants 2016-03424 and 621-2014-3974, respectively.
\end{ack}

\section{Preliminaries}

To keep this note short, we follow the notation from 
Bj\"orn--Bj\"orn--Latvala~\cite{BBLat3}, without repeating all
the discussion here; see 
\cite{BBLat3} for more references.
As usual, we assume that $1<p<\infty$
and that $X$ is a complete metric space equipped with a doubling measure
$\mu$ which supports a \p-Poincar\'e inequality, i.e\
there are constants $C,\la>0$ such that for all open  balls 
$B=B(x,r)$,
we have
\[
\mu(2B)\le C\mu(B)
\]
and, setting $u_{B}=\int_{B} u\,d\mu/\mu(B)$,
\begin{equation}    \label{eq-def-PI}
\frac{1}{\mu(B)} \int_{B} |u-u_B| \,d\mu
     \le C r \biggl( \frac{1}{\mu(B)} \int_{\la B} g^p 
               \,d\mu \biggr)^{1/p}
\end{equation}
holds for all integrable functions $u$ on $\la B:=B(x,\la r)$ and their
\p-weak upper gradients $g$. 
Here, $g: X \to [0,\infty]$ 
is a \emph{\p-weak upper gradient} of $u:X \to [-\infty,\infty]$ 
if for $\Modp$-almost every curve $\ga$ in $X$,
\begin{equation}  \label{eq-def-wug}
|u(x)-u(y)| \le \int_\ga g\,ds,
\end{equation}
where $x$ and $y$ are the end points of $\ga$ and the left-hand side
is interpreted as $\infty$ whenever at least one of the terms therein
is infinite. 
By ``holding for $\Modp$-almost every curve $\ga$'' we mean that there is 
$\rho\in L^p(X)$ such that $\int_\ga\rho\,ds=\infty$ for all $\ga$,
where \eqref{eq-def-wug} fails.
All curves considered here
are nonconstant, compact and rectifiable, and thus can
be parameterized by arc length $ds$.

Having defined the \p-weak upper gradients, the \emph{Newtonian Sobolev space}
$\Np (X)$ is defined as the collection of all $u\in L^p(X)$ 
having
a \p-weak upper gradient $g\in L^p(X)$.
Every $u \in \Np(X)$ 
has a \emph{minimal} \p-weak upper gradient
$g_u$ (well-defined up to sets of measure zero)
such that $g_u \le g$ for every \p-weak upper gradient $g \in L^p(X)$ of $u$.

$\Np(U)$ is defined similarly for 
arbitrary $U \subset X$, but in that case
\eqref{eq-def-wug} is only required for $\Modp$-almost
every curve $\ga$ within $U$. 
This is possible since \p-quasiopen sets are measurable, 
by Bj\"orn--Bj\"orn~\cite[Lemma~9.3]{BBnonopen}.
It was shown in  \cite[Proposition~3.5]{BBnonopen} 
that if $U$ is
\p-quasiopen then \p-weak upper gradients with respect to $U$ coincide
with those taken with respect to the whole space $X$.

Functions in $\Np (X)$ (and in $\Np(U)$ if $U$ is \p-quasiopen)
are precisely defined up to sets of zero \p-capacity,
which in turn is defined for an arbitrary $E\subset X$ as
\[
\Cp(E) = \inf_{u}  \int_X (|u|^p + g_u^p) \,d\mu,
\]
where the infimum is taken over all $u\in \Np(X)$ such that $u=1$ on $E$.

A set $V\subset X$ is \emph{\p-finely open} if $X\setm V$ is \p-thin at every 
$x\in V$, 
i.e.\  the Wiener type integral
\begin{equation}   \label{deff-thin}
\int_0^1\biggl(\frac{\cp(B(x,r)\setm V,B(x,2r))}
         {\cp(B(x,r),B(x,2r))}\biggr)^{1/(p-1)}  \frac{dr}{r}<\infty.
\end{equation}
Here,
the variational \p-capacity $\cp$ is  defined by
\[
\cp(E,A)=\inf_{u} \int_X g_u \, d\mu,
\]
with the infimum  taken over all $u\in \Np(X)$ such that 
$\chi_E \le u \le \chi_{A}$.

Since, under our assumptions, $\cp$ and $\Cp$ have the same zero sets,
it follows that 
\p-fine openness is preserved under removing sets of zero \p-capacity, as
such sets
do not influence the Wiener type integral~\eqref{deff-thin}.
This also shows that the complement of a set of zero \p-capacity is not
\p-thin at any  $x\in X$, so nonempty \p-finely open sets must have positive 
\p-capacity.

Adding sets of zero \p-capacity to \p-finely open sets 
does not necessarily preserve \p-fine openness, but it
produces \p-quasiopen sets:
By Theorem~1.4\,(a) in Bj\"orn--Bj\"orn--Latvala~\cite{BBLat2},
a set $U$ is \emph{\p-quasiopen} if and only if $U=V \cup E$, where
$V$ is \p-finely open and $\Cp(E)=0$. 
Typically, this decomposition is not unique.

The \p-finely open sets define the \p-fine topology, and a \p-finely open set
$U$ is \emph{\p-finely connected} if it is connected in this topology, 
i.e.\ it cannot be written as a disjoint union of nonempty \p-finely open sets.

For \p-quasiopen sets $U$, it is natural to define the \emph {local
Newtonian Sobolev space} $\Npploc(U)$, which consists of all functions 
$u:U\to[-\infty,\infty]$ such $u\in\Np(V)$ for every \p-finely open \p-strict 
subset $V\Subset U$.
Here, $V$ is a \emph{\p-strict subset} of $U$ if there is $u\in\Np(X)$ such
that $\chi_V\le u\le \chi_U$, or equivalently, if $\cp(V,U)<\infty$.
The space $\Lpploc(U)$ is defined similarly.

Every $u\in \Npploc(U)$ has a \emph{minimal} \p-weak upper gradient 
$g_u \in \Lpploc(U)$ 
(well-defined up to sets of measure zero) such that
$g_u \le g$ for every \p-weak upper gradient
$g \in \Lpploc(U)$ of $u$,
see Bj\"orn--Bj\"orn--Latvala~\cite[Section~5]{BBLat3}.
(The results in \cite{BBLat3} are for the even
larger space $\Npqloc(U)$, but can easily be adapted
to $\Npploc(U)$.)

For \p-quasiopen $U\subset\R^n$, the spaces $\Np(U)$ and $\Npploc(U)$
are essentially the Sobolev spaces $W^{1,p}(U)$ and $W^{1,p}\loc(U)$,
defined by Kilpel\"ainen--Mal\'y~\cite{KiMa92},
see the discussion after Corollary~\ref{cor-weakconn}
for more details.
(We remark that the space $\Npploc(U)$ is more natural in fine
potential theory than the smaller space $\Nploc(U)$ consisting of 
those functions $u$
such that for every $x \in U$ there is $r>0$ such that 
$u \in \Np(U \cap B(x,r))$.)

Similarly to $\Np(U)$, also functions in $\Npploc(U)$ are precisely 
defined up to sets of \p-capacity zero, as seen in the following lemma.

\begin{lem} \label{lem-qcont}
Let $U$ be \p-quasiopen and $u,v \in \Npploc(U)$.
If $u=v$ a.e.\ in $U$, then $u=v$ \p-q.e.\ in $U$,
i.e.\ $\Cp(\{x \in U : u(x) \ne v(x)\})=0$.
\end{lem}

\begin{proof}
Theorem~4.4 in Bj\"orn--Bj\"orn--Latvala~\cite{BBLat3} shows that $u$ and $v$ 
are \p-quasi\-con\-tin\-u\-ous, both with respect to $C_p$ and $\CpU$,
where $\CpU$ is
obtained by regarding $U$ as a metric space in its own right.
Proposition~5.23 in Bj\"orn--Bj\"orn~\cite{BBbook} (applied to $U$) 
then implies that $u=v$ $\CpU$-q.e.,
and thus also \p-q.e., since $C_p$ and $\CpU$ have the same zero sets (by 
Bj\"orn--Bj\"orn--Mal\'y~\cite[Proposition~4.2]{BBMaly}).
\end{proof}

\section{Proofs}

To prove Theorem~\ref{thm-main}, we first 
obtain the following result in general metric spaces.

\begin{thm} \label{thm-main-X}
Let $U\subset X$ be a \p-quasiopen set.
Then the following are equivalent\/\textup{:}
\begin{enumerate}
\renewcommand{\theenumi}{\textup{(\roman{enumi})}}%
\item \label{i-gu}
If $u \in \Npploc(U)$ and $g_u=0$ a.e.,
then there is a constant $c$ such that $u=c$ a.e.\ in $U$.
\item \label{i-quasiconn}
$U$ is \p-quasiconnected.
\item \label{i-finconn}
If $U=V \cup E$, where $V$ is \p-finely open and $\Cp(E)=0$,
then $V$ is \p-finely connected.
\end{enumerate}
\end{thm}

Note that in \ref{i-finconn}, the \p-fine connectedness should hold
for every decomposition of $U$.
Latvala's result (Theorem~\ref{thm-Latvala} below) shows that
in unweighted $\R^n$, it can equivalently be assumed only for some decomposition
$U=V\cup E$; we do not know if this is true in metric spaces,
nor in weighted $\R^n$.

It follows from the proof below that  
the space $\Npploc(U)$ in \ref{i-gu} can be replaced by the smaller
space $\Np\loc(U)$, or even the smaller space consisting of those 
functions $u$ such $u \in \Np(U \cap B)$ for every ball $B\subset X$
(since $\chi_{U_1}\in\Np(U\cap B)$ in that case).
If $U$ is bounded, then 
the space $\Np(U)$ can be used instead, but
this is not possible in general as can be seen by considering 
$U=\{(x_1,x_2) \in \R^n : x_1 \ne 0\}$, which is not \p-quasiconnected
and yet 
every $u \in \Np(U)$ with   $g_u=0$ a.e. 
must be a.e.-constant.

In \ref{i-gu} it is equivalent to require that $u=c$ \p-q.e.\ in $U$
by Lemma~\ref{lem-qcont}, 
while  \ref{i-finconn} can equivalently be formulated as follows:
If $\Cp(E)=0$ and $U\setm E$ is \p-finely open, then $U \setm E$ is \p-finely
connected.

If a \p-weak upper gradient of $u$ is modified on a set of measure zero 
it remains a \p-weak upper gradient of $u$. 
In particular, the condition $g_u=0$ a.e.\ in \ref{i-gu} is
equivalent to requiring that zero is a \p-weak upper gradient of $u$.
We will use this fact in the proof below.

\begin{proof}
$\neg$\ref{i-quasiconn} $\imp$ $\neg$\ref{i-finconn}
By assumption, there is a \p-quasiopen set $U_1 \subset U$
such that $U_2=U \setm U_1$ is also \p-quasiopen and in addition
$\Cp(U_j)>0$, $j=1,2$.
We can write $U_j=V_j \cup E_j$, where $V_j$ is \p-finely open
and $\Cp(E_j)=0$.
Then $\Cp(V_j)= \Cp(U_j)>0$, and thus $V_j$ is nonempty, $j=1,2$.
Letting $V=V_1 \cup V_2$ and $E=E_1 \cup E_2$ shows that
\ref{i-finconn} fails.

$\neg$\ref{i-finconn} $\imp$ $\neg$\ref{i-quasiconn}
Let $U=V \cup E$, where $\Cp(E)=0$ and $V$ is a \p-finely open set
which is not \p-finely connected.
Then $V=V_1 \cup V_2$ for some nonempty disjoint \p-finely open sets
$V_1$ and $V_2$.
Since $V_2 \cup E$ is \p-quasiopen, $V_1=U\setm(V_2\cup E)$ is relatively
\p-quasiclosed within $U$, as well as \p-quasiopen.
As $\Cp(V_j)>0$, $j=1,2$, it follows that $U$ is not 
\p-quasiconnected.

$\neg$ \ref{i-gu} $\imp$ $\neg$\ref{i-quasiconn}
Let $u \in \Npploc(U)$ be a function with 
$g_u=0$ a.e.\ and assume that there is  $m \in \R$ such that
$\Cp(\Upm)>0$, where 
\[
\Uplus=\{x \in U : u(x)>m\} 
\quad \text{and} \quad \Uminus=\{x \in U : u(x)\le m\}.
\]
Since zero is a \p-weak upper gradient of $u$, there are $\Modp$-almost
no rectifiable curves starting  in $\Uplus$ and ending in $\Uminus$.
Hence zero is a \p-weak upper gradient of $v=\chi_{\Uplus}$ defined
on $U$, and thus $v \in \Npploc(U)$.

By Theorem~4.4 in Bj\"orn--Bj\"orn--Latvala~\cite{BBLat3}, 
$v$ is \p-quasicontinuous.
Hence the level sets 
\[
\bigl\{x \in U : v(x) > \tfrac{1}{2}\bigr\}=\Uplus
\quad \text{and} \quad 
\bigl \{x \in U : v(x) < \tfrac{1}{2}\bigr\}=\Uminus,
\]
which together constitute $U$,
are \p-quasiopen, by 
Bj\"orn--Bj\"orn--Mal\'y~\cite[Proposition~3.4]{BBMaly}.
Since $\Cp(\Upm)>0$, \ref{i-quasiconn} fails.

$\neg$\ref{i-quasiconn} $\imp$ $\neg$\ref{i-gu}
As $U$ is not \p-quasiconnected, it
can be written as a union of two disjoint \p-quasiopen sets
$U_1$ and $U_2$ with positive \p-capacity.
We shall show that 
the characteristic function $u=\chi_{U_1}$ has zero as a \p-weak upper
gradient (within $U$), and consequently belongs to $\Npploc(U)$.
Since it is not \p-q.e.-constant, 
and thus not a.e.-constant either (by Lemma~\ref{lem-qcont}), 
this violates \ref{i-gu}.

To see that zero is a \p-weak upper gradient of $u$, it suffices
to show that there are $\Modp$-almost no curves within $U$ passing
from $U_1$ to $U_2$. 
As $U_1$ and $U_2$ are \p-quasiopen, Remark~3.5 in 
Shanmugalingam~\cite{Sh-rev} implies that they are \p-path open (within $U$),
i.e.\ for $\Modp$-almost every curve $\ga:[0,l_\ga]\to U$, the preimages
$\ga^{-1}(U_j)$, $j=1,2$, are relatively open (and disjoint) in 
\[
[0,l_\ga] = \ga^{-1}(U_1)\cup \ga^{-1}(U_2).
\]
But this is impossible, so there are no such curves.
\end{proof}

The following direct consequence of the equivalence 
\ref{i-quasiconn} $\eqv$ \ref{i-finconn} in Theorem~\ref{thm-main-X}
motivates our terminology.

\begin{cor} \label{cor-weakconn}
Every \p-quasiopen \p-quasiconnected set $U \subset X$ is weakly 
\p-quasi\-con\-nec\-ted.
\end{cor}

The following result about preserving \p-fine connectedness
was proved by Latvala~\cite{Lat2000}
in unweighted
$\R^n$, while it still remains open in more general situations 
(including $\R^n$ with \p-admissible weights). 

\begin{thm} \label{thm-Latvala}
\textup{(Latvala~\cite[Theorem~1.1]{Lat2000})}
Let $V\subset\R^n$ \textup{(}unweighted\/\textup{)}
be \p-finely open and \p-finely connected,
$1<p\le n$. 
If $\Cp(E)=0$ then $V\setm E$ is also \p-finely connected\/ 
\textup{(}and \p-finely open\/\textup{)}.
\end{thm}

(A similar statement with $p>n$ is trivial, since in that
case the \p-fine topology is just the usual Euclidean one.)
The converse implication is much easier
and holds in general metric spaces satisfying our assumptions:

\begin{lem} \label{lem-Lat-converse}
Let $V\subset X$ be \p-finely open and assume that $V\setm E$ is
\p-finely connected for some $E$ with $\Cp(E)=0$.
Then $V$ is also \p-finely connected.
\end{lem}

\begin{proof}
Assume  that $V$ is not \p-finely connected, i.e.\ it can be written as
$V=V_1\cup V_2$, where $V_1$ and $V_2$ are nonempty disjoint \p-finely open sets.
In particular, they have positive \p-capacity.
Then $V_j\setm E$ are also \p-finely open and nonempty, so 
$V\setm E= (V_1\setm E)\cup (V_2\setm E)$ cannot be 
\p-finely connected.
\end{proof}

As already mentioned, the
 Newtonian Sobolev space $\Np$ is more precisely defined than
the traditional Sobolev spaces $\Wp$ on $\R^n$.
For a \p-quasiopen set $U$ in unweighted $\R^n$, 
the space denoted $\Wploc(U)$ 
in Kilpel\"ainen--Mal\'y~\cite{KiMa92} coincides with the space
\[
        \hNpploc (U) = \{u : u=v \text{ a.e. for some } v \in \Npploc(U)\},
\]
see Bj\"orn--Bj\"orn--Latvala~\cite[Theorem~5.7]{BBLat3}.
Similarly, functions in $\Wp(U)$ are a.e.\ equal to functions
from $\Np(U)$.
Moreover, the modulus  of
the \p-fine gradient $\nabla u$, introduced in \cite{KiMa92},
coincides a.e.\ with the minimal \p-weak 
upper gradient $g_v$ of $v$, i.e.\ $g_v=|\nabla u|$ a.e., 
see \cite[Theorems~5.3 and~5.7]{BBLat3}.
The situation is similar in weighted $\R^n$ with a \p-admissible weight, 
provided that $\nabla u$ stands for the corresponding weighted \p-fine
gradient; cf.\  the discussion in 
Heinonen--Kilpel\"ainen--Martio~\cite[p.\ 13]{HeKiMa}.
Thus on unweighted and
weighted $\R^n$ (with a \p-admissible weight),
\ref{i-gu} in Theorem~\ref{thm-main-X}  
is equivalent to \ref{a-gu} in Theorem~\ref{thm-main}.
We are now ready to prove our main result.

\begin{proof}[Proof of Theorem~\ref{thm-main}]
\ref{a-gu} $\eqv$ \ref{a-conn} $\eqv$ \ref{a-finconn}
These equivalences follow from Theorem~\ref{thm-main-X},
in view of the discussion above.

\ref{a-conn} $\imp$ \ref{a-weakconn}
This follows from Corollary~\ref{cor-weakconn}.

$\neg$\ref{a-finconn} $\imp$ $\neg$\ref{a-weakconn}
By assumption, 
$U=V\cup E$, where $\Cp(E)=0$ and $V$  is \p-finely open but 
not \p-finely connected.
Let $U=V'\cup E'$ be any other decomposition of $U$ into a \p-finely open
set $V'$ and a set $E'$ with $\Cp(E')=0$.
Lemma~\ref{lem-Lat-converse}
then implies that $V'\setm (E\cup E')=V\setm (E\cup E')$ is not
\p-finely connected.
An application of Latvala's theorem~\ref{thm-Latvala} then shows that
$V'$ is not \p-finely connected either, i.e.\ \ref{a-weakconn} fails.
\end{proof}

\end{document}